\documentclass[12pt]{article}

\usepackage{a4wide}
\usepackage{longtable}
\usepackage{tabu}
\usepackage{graphicx}
\usepackage{amsfonts}
\usepackage{pdfsync}
\usepackage{amsmath}
\usepackage{enumerate}
\usepackage{amsthm}

\usepackage{pdfsync}

\usepackage[colorlinks=true,linkcolor=blue,citecolor=red]{hyperref}

\newcommand{\R}{\mathbb{R}}
\renewcommand{\L}{\mathbb{L}}

\newcommand{\ep}{\varepsilon}

\newcommand{\dt}{\partial_t}

\newcommand{\grad}{\nabla} 
\newcommand{\produ}[1]{\left\langle #1\right\rangle}

\newcommand{\nab}{\bar\nabla}

\newcommand{\SOf}{SO^{\uparrow}(n-1,1)}

\newtheorem{theorem}{Theorem}[section]
\newtheorem{lemma}{Lemma}[section]
\newtheorem{proposition}{Proposition}[section]
\newtheorem{corollary}{Corollary}[section]
\newtheorem{definition}{Definition}[section]
\newtheorem{examp}{Example}[section]
\newtheorem{remark}{Remark}[section]

\title{Translating Solitons in a Lorentzian Setting, Submersions and Cohomogeneity One Actions}

\author{
Marie-Am\'elie Lawn \\
\small Department of Mathematics, Faculty of Natural Sciences, \\
\small Imperial College London, 180 Queen‘s gate London SW72AZ (UK)\\
\small m.lawn@imperial.ac.uk\\
Miguel Ortega\\
\small Department of Geometry and Topology, Institute of Mathematics \\
\small Universidad de Granada, 18071 Granada (Spain)\\
\small miortega@ugr.es, ORCID 0000-0002-1390-9980
}
\date{\today}

\begin{document}
\maketitle
\begin{abstract} 
We study new examples of translating solitons of the mean curvature flow, especially in Minkowski space. We consider for this purpose manifolds admitting submersions and cohomegeneity one actions by isometries on suitable open subsets. This general setting also covers the classical Euclidean examples. As an application, we completely classify timelike, invariant translating solitons by rotations and boosts in Minkowski space. 
\end{abstract}


\section{Introduction}

The evolution by mean curvature flow is classicaly studied for hypersurfaces in the Euclidean space 
$\mathbb{R} ^{n+1}$. One of the approaches is to consider particular solutions, the so-called translating solitons or translators, which are invariant by a subgroup of translations of the ambient space. It is well-known that they admit a forever flow, characterized by a constant (unit) vector $v\in\mathbb{R}^{n+1}$, $n\geq 1$, in such a way that the mean curvature flow equation simplifies to 
\[H=v^{\perp}.\]
Translators have been widely studied in literature (\cite{AW, CSS,MSHS, LTW},\dots, and references therein). For instance, they naturally appear in the study of solutions of the mean curvature flow with a certain type of singularities (see for example \cite{HS}) and are equivalent to minimal surfaces for a conformally modified metric \cite{Il}. There are other studies for translators in other ambients spaces, such as $\R\times M$ \cite{LM}, in $\mathbb{H}^2\times\mathbb{R}$ \cite{Bueno}, a solvable group \cite{Pipoli}, the Heisenberg 3-group \cite{Pipoli2},  etc.

One of the first analytical point of view was to see which of them were rotationally invariant \cite{CSS}. This approach can be revisited from a geometrical point of view, in the sense that they are invariant by the Lie group $SO(n-1)$. Thus, the PDE is rapidly reduced to an ODE, which is simpler to study. Translating solitons were studied from an analytical point of view in the Lorentz-Minkowski space, \cite{coreano}, focusing on the rotationally invariant case and only for spacelike hypersurfaces. Again, the underlying geometrical structure is the use of the Lie group $SO(n)$ as a subgroup of isometries. 

In this paper, we are interested in studying graphical translating solitons in Minkowski space  by the action of some subgroups of isometries. From a pure analytical point of view, we seek to simplify the PDE to an ODE by a group of transformations. However, adding a geometrical sauce  provides an extra layer of flavour. To do so, we wished to use quotients, although we realized that suitable tools are submersions, for almost the same effort. Those who wish to stick to the analytical taste can jump directly to Sections \ref{Rotational} and \ref{SOf}.

Typical examples are cohomogeneity one actions by Lie groups. If the action of  a group $\Sigma$ is proper and free, the orbit space becomes a smooth manifold (with boundary), and the natural projection is a submersion. This holds when  $M$ is Riemannian and $\Sigma$ is compact, acting by isometries. 
However,  the situation becomes much more complicated when the action is not proper.
Fortunately,  it is still possible to work on some good open subsets $\Omega$ of $M$ quite often, because the orbit space $\Omega/\Sigma$ is diffeomorphic to an open interval, even when $M/\Sigma$ is not  Hausdorff.

Since our starting point was to study translators in Lorentz-Minkowski space, the base manifolds of the quotients have to be either Riemannian or Lorentzian. To unify the notations, we are forced to consider  pseudo-Riemannian submersions. The advantage is that we can make a quite more general study, although we will return to our origins in the last sections.

We will mainly focus on \textit{graphical} translating solitons. Namely, given $u\in C^2(M)$, we construct its graph map $F:M\rightarrow M\times\R$, $F(x)=(x,u(x))$, and we assume that its induced metric is not degenerate.  Let $\nu$ be the upward normal vector along $F$ with $\ep=\produ{\nu,\nu}=\pm 1$. We recall that $\ep=-1$ when $F(M)$ is spacelike, and $\ep=+1$ if timelike. In Section \ref{setup}, we show  that the mean curvature flow is characterized by the following partial differential equation, which is similar to the one obtained in the classical case (see \cite{CSS})
\[
\mathrm{div}\left(\frac{ \grad u } { \sqrt{\ep\big(\ep'+\vert \grad u\vert_g^2\big)}} \right)=
\frac{1}{ \sqrt{\ep\big(\ep'+\vert \grad u\vert_g^2\big)}} =  H. 
\]
As in \cite{Bueno,CSS,LM}, we show  in Section \ref{subslie} that it is possible to simplify this PDE to an ODE under reasonable  assumptions on the geometry of the manifold $M$, namely when $M$ admits a submersion $\pi:(M,g)\rightarrow (I,\tilde\ep ds^2)$ whose fibers $\pi^{-1}(s)$ have constant mean curvature $h(s)$, where $I$ is an interval, $\tilde{\ep}=\pm 1$. In Theorem \ref{equiv_ODE}, we prove that a function $f\in C^2(M)$ is solution to the ODE
\[f''(s)=(\tilde{\ep}+\ep' f'(s)^2)(1-f'(s)h(s)),\] 
if, and only if, the graph map $F:M\to M\times\R$, $F(x)=(x,(f\circ\pi)(x))$, is a translating soliton. 
Next, we will give some results about the behavior of translating solitons within this context.  Corollary \ref{oide} generalizes the well-known Wing-like translating soliton in $\mathbb{R}^{m+1}$, \cite{CSS}, whose tangent planes at some points are vertical, and thus, they are not graphical at such points. The standard idea is to rewrite the ODE by putting $f$ as the inverse function of a given $\alpha$. Corollary \ref{global-solutions} is an existence result for globally defined graphical solutions.

Section \ref{soluODE} is devoted to analytical computations, which are necessary for the final sections.  

Section \ref{Rotational} is devoted to applying all these ideas to rotationally invariant translating solitons in Minkowski space $\L^{n+1}$, $n\geq 2$, that is to say, those invariant by the action of $SO(n)$ as a subgroup of isometries of $\L^{n+1}$. Also, in Section \ref{SOf}, we consider the Lie group $SO^{\uparrow}(n-1,1)$ acting on $\L^{n+1}$, $n\geq 2$. We focus on the timelike cases, because the spacelike cases were already considered in \cite{coreano} (although without the Geometrical topping). Probably, the most surprising result is the fact that there is an example which crosses more than one fundamental region. Or in other words, we need more than one profile curve to obtain the whole surface, by gluing up four pieces lying in continuous open subsets. Probably, this unusual and slippery technique seems to prevent many people from studying timelike examples, but makes the Lorentzian setting surprisingly different from the classical Euclidean case. 

\color{black}
\section{Setup}\label{setup}

Let $(M,g)$ be a connected\, pseudo-Riemannian manifold. Consider $u\in C^2(M,\R)$, and let $F:M\rightarrow M\times\R=:\bar{M}$, $F(p)=(p,u(p))$\, be  its graph map. We denote $(p,t)\in M\times\R$. We denote by $\vert w\vert_g^2={g(w,w)}$ the squared $g$-norm of any tangent vector $w$ to $M$. Assume that $F:(M,\gamma=F^*\produ{,})\rightarrow (\bar{M},\produ{,}=g+\ep'dt^2)$ is a non-degenerate hypersurface, where $\ep'=\pm 1$.  Under the usual identifications, for each $X\in TM$, we have
\[ dF(X) = (X,du(X)) = (X,g(\grad  u,X)), \]
where $\grad  u$ is the $g$-gradient of $u$.  We consider the metric $\gamma=F^*\produ{,}$ on $M$. By assumption, $\gamma$ is non-degenerate. The upward  normal vector field is
\begin{equation}\label{normal2} 
\nu = \frac{1}{W}(-\ep'\grad u,1), \quad W=+\sqrt{\ep\big(\ep'+\vert\grad u\vert_g^2\big)},
\end{equation}
where we are assuming that $\ep:=\mathrm{sign}\big(\ep'+\vert\grad u\vert_g^2\big)=\pm 1$ is a constant function on the whole $M$.
Note that $\produ{\nu,\nu}=\varepsilon$. 
It is important to remark our definitions of the mean curvature vector and function. In our setting, if $II_F$ is the second fundamental form, the mean curvature vector is 
\[ \vec{H}_F=\mathrm{trace}_{g_M}(II_F)=\ep H \nu_F,
\]
where $H$ is the mean curvature function. Similar notations will be used for other hypersurfaces along the paper.

The following proposition is well-known in the Euclidean setting (see for example \cite{LTW}). The proof is similar in the pseudo-Riemannian context, but we include it for completeness and discuss the necessary modifications.
\begin{proposition}\label{basicequation}
Under the previous setting, $F$ is a (vertical) translating soliton  if, and only if, function $u$ satisfies \begin{eqnarray}\label{fund_eq}
\mathrm{div}\left(\frac{ \grad u } { \sqrt{\ep\big(\ep'+\vert \grad u\vert_g^2\big)}} \right)=
\frac{1}{ \sqrt{\ep\big(\ep'+\vert \grad u\vert_g^2\big)}} = H. 
\end{eqnarray}
\end{proposition}
\begin{proof}
Take a local $g$-orthonormal frame $B=(e_1,\ldots,e_n)$ on $M$, such that $g(e_i,e_j)=\ep_i\delta_{ij}$ for any $i,j=1,\dots n$, with $\ep_i=\pm 1$ depending on the signature of $g$. We denote $u_i=du(e_i)$, $i=1,\ldots,n$. For this frame, we compute the induced metric $\gamma=F^{*}\produ{,}$, so that the coefficients of the Gram matrix and of its inverse are 
\[ \gamma_{ij}=\gamma(e_i,e_j) = \ep_i\delta_{ij}+ \ep'u_iu_j, \quad  
 \gamma^{ij} = \ep_i\delta_{ij} -\frac{\ep}{W^2}\ep_i\ep_ju_iu_j.
\]
Let  $\nab$ the Levi-Civita connection of $(\bar{M},\produ{,})$. We recall O'Neill's book \cite{ON}, and its equations for the Levi-Civita connection of a (warped) product. Then, 
\begin{align*}
&H=\produ{\vec{H}_F,\nu}=\produ{\dt^{\perp}, \nu}=\produ{\dt,\nu}  =\produ{\nu,\mathrm{tr}_\gamma(II_F)} 
=\sum_{i,j}\gamma^{ij}\produ{\nu,\nab_{dF(e_i)}dF(e_j)} \\
& = \sum_{i,j}\gamma^{ij}\produ{\nu,\big(\nabla_{e_i}e_j,0\big)+e_i(u_j)\dt } 
=\sum_{i,j}\gamma^{ij}e_i(u_j)\produ{\nu,\dt},
\end{align*}
so that
\begin{align*}  1=\sum_{i,j}\gamma^{ij}e_i(u_j) = \sum_{i,j}\left(\ep_i \delta_{ij} -\frac{\ep_i\ep_ju_iu_j}{\ep W^2}\right)e_i(u_j) = \mathrm{div}(\nabla u) -\frac{\ep}{W^2}\sum_{i,j}\ep_i\ep_ju_iu_je_i(u_j).
\end{align*}
By recalling that $W^2=\ep(\ep'+g(\nabla u,\nabla u))$, we compute
\begin{align*}
& 2W\produ{\nabla W,\nabla u}=\produ{\nabla (W^2),\nabla u}=\ep \sum_{i}\ep_iu_ie_i\left(g(\nabla u,\nabla u)\right)
 \\
& = \ep \sum_{i} \ep_i u_i \sum_j\ep_j e_i(u_j^2) = 2\ep\sum_{i,j} \ep_i\ep_j u_iu_je_i(u_j).
\end{align*}
Coming back
\begin{gather*}
1=\mathrm{div}(\nabla u)- \frac{1}{W}\produ{\nabla W,\nabla u},\\
\frac{1}{W}=\frac{\mathrm{div}(\nabla u)}{W} -\frac{1}{W^2}\produ{ \nabla W,\nabla u} 
=\frac{\mathrm{div}(\nabla u)}{W} + \produ{\nabla\big(1/W\big),\nabla u} =\mathrm{div}\left(\frac{\nabla u}{W}\right).
\end{gather*}
\end{proof}
\begin{corollary} Let $(M,g)$ be a compact orientable manifold without boundary.  Then, $M$ does not admit any globally defined, non-degenerate, graphical, spacelike Translating Soliton $F:M\rightarrow (M\times\mathbb{R},g+\ep'dt^2)$.
\end{corollary}
\begin{proof} Assume that there exists a globally defined graphical Translating Soliton on $M$. Then, for some function $u\in C^2(M)$,  equation \eqref{fund_eq} holds true. By using the volume form $d\mu_g$, we obtain
\[ 0 = 
\int_M \mathrm{div}\left(\frac{\nabla u}{\sqrt{\varepsilon(\ep'+\vert \nabla u\vert^2_g)}}\right) d\mu_g = \int_M \frac{1}{\sqrt{\varepsilon(\ep'+\vert \nabla u\vert^2_g)}} d\mu_g >0.
\]
This is a contradiction.
\end{proof}

\section{Submersions and Lie Groups}\label{subslie}

Our next target consists of studying when the equation of the mean curvature flow can be reduced to a particular ODE. The main geometrical technique consists of using a Lie group acting by isomorphisms whose orbits are codimension one sumbanifolds, namely cohomogenity one actions. We choose the case when the image of the natural projection is going to be identified with an open interval. This is not a big deal, since one can remove the non-regular orbits, but later it might be necessary to see the possible extension. Before starting, we need the following technical lemma.

\begin{lemma} \label{lemilla} Take $\ep',\tilde{\ep}\in\{\pm 1\}$, $I$  an open interval and $\pi:(M,g_{_M})\rightarrow(I,\tilde\ep ds^2)$ a pseudo-Riemannian submersion such that each fiber $\pi^{-1}(s)$ has constant mean curvature $h(s)$ w.r.t. $-\grad\pi$.  Given $f\in C^2(I)$,  define $u=f\circ\pi\in C^{2}(M)$. 
Assume that $\ep=\mathrm{sign}(\ep'+ \tilde\ep (f'\circ \pi)^2)=\pm 1$ is a constant function. 
\begin{enumerate}
\item $h:I\rightarrow\R$ is a smooth function, and such that $\mathrm{div}(\grad\pi)=h\circ\pi$.
\item $\grad u = (f'\circ \pi)\grad \pi$, $\vert \grad u\vert_g^2=  \tilde\ep (f'\circ\pi)^2$, $\mathrm{div}(\grad u) =  \tilde\ep (f''\circ\pi)+(f'\circ\pi)(h\circ\pi)$.
\item If $W=+\sqrt{\ep\big(\ep'+ \tilde\ep(f'(\pi))^2\big)}\,>0$, then $\grad W =\displaystyle \frac{\ep \tilde\ep f'(\pi)f''(\pi)}{W}\grad\pi.$ 
\end{enumerate}
\end{lemma}
\begin{proof} First, given a local orthonormal frame $(e_1,\ldots,e_n)$ such that $e_n=\grad\pi$ and the vector fields $\{e_i\}_{i=1}^{n-1}$ are tangent to the fibers, with $\ep_i=g(e_i,e_i)$,  then   
\[\mathrm{div}(\grad\pi)=\sum_{i=1}^n \ep_i g(\nabla_{e_i}\grad\pi,e_i) 
=\sum_{i=1}^{n-1} \ep_i g(\nabla_{e_i}\grad\pi,e_i) =h\circ \pi.
\]
For simpleness, we write $f\circ\pi=f(\pi)$, etc. For the second statement, given $p\in M$ and $X\in T_pM$, we compute  
\[g_p(\grad(f\circ\pi),X)=d(f\circ\pi)_pX = df_{\pi(p)}(d\pi_p(X)) = f'(\pi(p))g_p(\grad\pi,X),\] 
and we point out that $|\nabla\pi|_g^2=\tilde\ep$. Next, 
\[\mathrm{div}(\grad u)=\mathrm{div}(f'(\pi)\grad\pi) = g(\grad (f'\circ\pi),\grad \pi) + f'(\pi)\mathrm{div}(\grad\pi)
 = \tilde\ep f''(\pi)+f'(\pi)h(\pi). 
\]
Finally, $2W\grad W = \grad(W^2) = 2\tilde\ep\ep  f'(\pi)f''(\pi)\grad\pi.$ 
\end{proof}

\begin{proposition}\label{prop_Liegroup_is_submersionwithconstant mean curvaturefibers}
Let $(M,g)$ be a connected pseudo-Riemannian\,manifold and $I$ an open interval. Let $\Sigma$ be a Lie group acting by isometries on $M$ and $\pi:M\rightarrow  I$ be a submersion such that the fibers of $\pi$ are orbits of the action, and $\grad\pi$ is never zero or lightlike.  Then, there exist a constant $\tilde{\ep}=\pm 1$ and a smooth bijective map $v:I\rightarrow J\subset \R$, $J$ an interval, such that $v\circ\pi:(M,g)\rightarrow(J,\tilde\ep ds^2)$ is a pseudo-Riemannian\,submersion with constant mean curvature fibers. 
\end{proposition}
\begin{proof} 
Since $\pi$ is a submersion, and $\vert\grad \pi\vert_g^2\neq 0$, there exists a smooth function $z:I\rightarrow\R\backslash\{0\}$ such that $(z\circ \pi)^2 =\tilde\ep \vert\grad\pi\vert_g^2$, for a constant $\tilde \ep=\pm 1$. Next, we consider a function $v:I\rightarrow \R$ such that $v'=1/z$. As $v'$ has constant sign, $v$ is injective, and we restrict it to its image, $v:I\rightarrow J$. By Lemma \ref{lemilla}, $\vert \grad(v\circ\pi)\vert^2_g =\tilde\ep$.  In particular, $v\circ \pi:(M,g)\rightarrow (J,\tilde\ep ds^2)$ is a pseudo-Riemannian\,submersion. Note that the fibers of $v \circ \pi$ are the same as the fibers of $\pi$, which are  the orbits of $\Sigma$. As it acts by isometries, they have constant mean curvature w.r.t. $-\grad(v\circ\pi)$. 
\end{proof}

From now on, we assume that the action of the Lie group $\Sigma$ on $(M,g)$ by isometries is \textit{proper}, and that at least one of the orbits is of codimension one. These are the well-known \textit{cohomogeneity one $\Sigma$-manifolds}. According to \cite{AA} (and references therein), there is another Riemannian metric $g_R$ such that $\Sigma$ acts on $(M,g_R)$ by isometries, and the quotient $M/\Sigma$ is diffeomorphic to an interval or to $\mathbb{S}^1\equiv \R/\mathbb{Z}$. By removing the possible singular orbits (if any, up to two), we can assume that $M/\Sigma=:I$ is (diffeomorphic to) an open interval. This works at  differentiable level, so we can return to our original metric $g$, and the quotient is still an open interval. Let $\pi:M\to I\equiv M/\Sigma$ be the projection. We need $\nabla\pi$ not to be zero or lightlike. By Lemma  \ref{prop_Liegroup_is_submersionwithconstant mean curvaturefibers}, we can assume that $g_{_M}(\nabla\pi,\nabla\pi)=\pm 1=:\tilde{\ep}$, making $\pi$ a pseudo-Riemannian submersion. Each fiber of $\pi$ will be an orbit of $\Sigma$. And because it acts by isometries, then each orbit $\pi^{-1}\{s\}$ will be of constant mean curvature $h(s)$ w.r.t $-\nabla\pi$. Thus, we have an associated smooth function $h:I\to\R$. We need to bear in mind all this setting, so that we introduce the following definition.

\begin{definition}\label{decomposable} Under the previous setting, let $\Sigma_o$ one of the stabilizers of $\Sigma$. Let $\beta:I\to M$ be a unit curve orthogonal to all orbits of $\Sigma$, such that the map 
\begin{equation} \label{equ:decom} \phi : I\times (\Sigma/\Sigma_o)\to M, \quad \phi(s,[\sigma])=\sigma\cdot \beta(s),
\end{equation}
is a diffeomorphism and $\pi\circ\phi(s,[\sigma])=s$. We will say that \eqref{equ:decom} is a decomposition of $M$ by $\Sigma$, or that $M$ is decomposable by $\Sigma$, with associated function $h:I\to\R$. We cannot forget the submersion $\pi:(M,g)\to (I,\tilde{\ep}ds^2)$. 
\end{definition}
\begin{remark}  For the sake of simpleness, sometimes is it better to use $\phi:I\times \Sigma\to M$. But then, $\phi(s,-):\Sigma\to M$ might not be an immersion, although the image is again an orbit.$\Box$ 
\end{remark}
We recall that a function $u:M\rightarrow\R$ is called \textit{invariant} by the Lie group $\Sigma$, or also $\Sigma$-invariant,  if it satisfies
\begin{equation}
\label{invariance} u:M\rightarrow\R, \quad u(x)=u(\sigma\cdot x), \ \forall x\in M, \ \forall \sigma\in \Sigma. 
\end{equation}
Accordingly, we say that a graphical translating soliton is \textit{invariant} by the Lie group $\Sigma$, or also \textit{$\Sigma$-invariant}, when its graph map is invariant by $\Sigma$. Similarly, a translating soliton $F:\Omega\subset M\to M\times\mathbb{R}$ is called $\Sigma$-invariant when $\sigma\cdot F(x)\in F(\Omega)$ for each $x\in \Omega$ and each $\sigma\in \Sigma$. 
\begin{theorem}\label{equiv_ODE} Let $(M,g)$ be a pseudo-Riemannian manifold, decomposable by $\Sigma$, as in Definition \ref{decomposable}. Given $\ep'=\pm 1$, $f:I\to \R$ a smooth function, construct the graph map $F:M\rightarrow (M\times\R,\produ{,}=g_{_M}+\ep'dt^2)$, $F(x)=(x,u(x))$, $u=f\circ\pi$.  Then, $F$ is a $\Sigma$-invariant translating soliton if, and only if,  $f$ is a solution to 
\begin{equation}
\label{keyODE} f''(s)=\big( \tilde\ep +\ep'  (f'(s))^2\big)\big(1-f'(s)h(s)\big).
\end{equation}
\end{theorem}
\begin{proof}[Proof of Theorem \ref{equiv_ODE}] We denote $u=f\circ\pi$. 
Using Lemma \ref{lemilla} we get directly from equation \eqref{basicequation}, 
\begin{align*}
& \frac{1}{W}=\mathrm{div}\left( \frac{\grad u}{W}\right) 
= \frac{\mathrm{div}(\grad u)}{W}- \frac{g(\grad W,\grad u)}{W^2} \\
& 
= \frac{\tilde\ep f''(\pi)+f'(\pi)h(\pi)}{W}- \frac{1}{W^2} g\left(\frac{\tilde\ep\ep f'(\pi)f''(\pi)}{W}\grad\pi,f'(\pi)\grad\pi\right) \\
& =\frac{ \tilde\ep f''(\pi)+f'(\pi)h(\pi)}{W} 
-  \frac{\tilde\ep\ep f'(\pi)^2f''(\pi)}{W^3} g\left(\grad\pi,\grad\pi\right) \\
& = \frac{1}{W} \left( f'(\pi)h(\pi)+  f''(\pi)  \frac{  \ep'\tilde\ep  }{\ep' + \tilde\ep f'(\pi)^2}  \right).
\end{align*}
Therefore, we conclude that \eqref{fund_eq} reduces to
\[  
f'(\pi)h(\pi) + \frac{  1 }{\tilde\ep +\ep 'f'(\pi)^2} f''(\pi) = 1.
\] 
In other words, function $f$ is a solution to \eqref{keyODE}. 
\end{proof} 
\begin{remark} Another approach is to start with a $\Sigma$-invariant function $u:M\to\R$. But then, we immediately obtain another function $f:I\to\R$ such that $u=f\circ\pi$.
\end{remark}

On the other hand, we can extend the action of $\Sigma$ to $M\times\R$ as follows: 
\begin{align} \nonumber 
\bar{M}=M\times\R, \produ{\,,}=g+\ep\mathrm{d}t^2,\ \ep=\pm 1, \quad \Sigma\times \bar{M}\to \bar{M}, \ 
(\sigma, (x,t))\mapsto (\sigma\cdot x,t). 
\end{align}
In fact, given any curve $\delta:J\subset I\to I\times\R$, $\delta=(\delta_1,\delta_2)$, we can construct a map whose image is $\Sigma$-invariant, namely
\begin{equation}\label{resureccion}
 F:J\times(\Sigma/\Sigma_o)\to M\times\R, \quad F(s,[\sigma])=\big(\phi(\delta_1(s),[\sigma]),\delta_2(s)\big).
\end{equation}
We will now give some general results about the behavior of solutions. This next result states the existence of the so-called \textit{wing-like solutions}. 
\begin{corollary} \label{oide} Assume there $M$ is decomposable as in Definition \ref{decomposable}, where $h$ is the mean curvature of the orbits of the action. Consider $s_o\in I$ such that either $h(s_o)\neq 0$, or $s_o$ is an isolated zero of $h$. Then, for each $y_o\in \R$, there exist a real number $\rho>0$ and a  translating soliton $F:(y_o-\rho,y_o+\rho)\times(\Sigma/\Sigma_o) \rightarrow\bar{M}$ such that it is the union of two graphical translating solitons. 
\end{corollary} 
\begin{proof} Choose $y_o\in\R$. We consider the following IVP: 
\[ \alpha''(y) = \left(\ep' +\tilde\ep\alpha'(y)^2\right) \left( h(\alpha(y)) - \alpha'(y)\right), \quad \alpha'(y_o)=0, \ \alpha(y_o)=s_o\in I.
\]
As usual, there exists a smooth solution $\alpha:(y_o-\rho,y_o+\rho)\rightarrow\R$. Note that $y_o$ is a critical point of $\alpha$ and $\alpha''(y_o)=\ep'h(s_o).$ \\
\noindent\textit{Case $h(s_o)\neq 0$:} Then, $y_o$ is  an extremum of $\alpha$. The restrictions $\alpha_+=\alpha\vert_{(y_o,y_o+\rho)}$ and $\alpha_{-}=\alpha\vert_{(y_o-\rho,y_o)}$ will be injective, by reducing $\rho$ if necessary. Construct their inverse functions $f_+=\alpha_+^{-1}$ and $f_{-}=\alpha_{-}^{-1}$. We just have to show that $f_+$ and $f_{-}$ satisfy \eqref{keyODE}. To do so, we put $f_{+}(\alpha(y))=y$, and therefore
\begin{align*}
& 1=f_+'(\alpha(y))\,\alpha'(y), \quad 0=f_+''(\alpha(y))\,\alpha'(y)^2+f_+'(\alpha(y))\,\alpha''(y), \\
& f_+''(\alpha(y))\,\alpha'(y)^2 = -f_+'(\alpha(y))\big(\ep'+\tilde\ep\alpha'(y)^2\big)\big(h(\alpha(y)-\alpha'(y)\big).
\end{align*}
Next, we change $s=\alpha(y)$, and then $\alpha'(y)=1/f'_+(s)$, so that 
\begin{align*}
\frac{f_+''(s)}{f_+'(s)^2} = -f_+'(s)\left(\ep'+\frac{\tilde\ep}{f_+'(s)^2}\right)\left( h(s) -\frac{1}{f_+'(s)}\right) = 
\frac{1}{f_+'(s)^2}\big(\tilde\ep+\ep'f_+(s)^2\big)\big(1-h(s)f_+(s)\big).
\end{align*}
A similar computation holds for $f_{-}$. The union of the corresponding graphical translating solitons and their common boundary provide a smooth translating soliton, because $\alpha$ is a smooth map and $f_{+}$, $f_{-}$ are tools to reparametrize its graph. \smallskip

\noindent \textit{Case $s_o$ is an isolated zero of $h$:} By shrinking $\rho$ if necessary, then $\alpha''(y)\neq 0$ for any $y\neq y_o$, $y\in (y_o-\rho,y_o+\rho)$. The restriction $\alpha'\vert_{[y_o,y_o+\rho)}$ will be injective, and therefore, $\alpha'(y)\neq 0$ for any $y\in(y_o,y_o+\rho)$. This makes $\alpha\vert_{(y_o,y_o+\rho)}$ also injective. Similarly, $\alpha\vert_{(y_o-\rho,y_o)}$ is injective. We continue as in the previous case $(h(s_o)\neq 0)$.\smallskip 

In either case,  it is possible to obtain a smooth curve $\tau:(y_o-\rho,y_o+\rho)\to I\times \R$, $\tau=(\tau_1,\tau_2)$, which is the union of both graphs. Then, our translating soliton is constructed by using \eqref{resureccion}.
\end{proof}

\begin{remark} By \eqref{equ:decom}, $\bar{M}$ is diffeomorphic to $I\times(\Sigma/\Sigma_o)\times\R$, so the profile curve $\tau$ of Corollary \ref{oide} can be  embedded into $I\times \{[\sigma]\}\times \R$ for some $[\sigma]\in (\Sigma/\Sigma_o)$. At a certain point, $\dt$ ($t\in\R$) will be tangent to the image of $\tau$ (as in the proof of previous lemma). The causal character of the product will therefore influence the  causal character of the translating soliton. Moreover the causal character of the submersion (i.e. the sign of $\tilde{\ep}$) as well as the sign of $\ep'+\tilde\ep f'^2$ will determine the causal character of the winglike soliton.  For example, when $M$ is Riemannian, the product $M\times\mathbb{R}$ is Lorentzian (i.e. $\ep'=-1$), and the submersion is Riemannian (as in the rotationally invariant example of Section \ref{Rotational}), the wing-like translating solitons have to be timelike, since for spacelike solitons, $\alpha'(y_0)$ cannot be $0$ for any $y_0$.
\color{black}$\Box$ 
\end{remark}

\begin{corollary}  \label{global-solutions} Assume that $M$ is decomposable as in Definition \ref{decomposable}. Take  $\tilde{\ep},\ep'\in\{\pm1\}$ such that  $\tilde{\ep}\ep'=-1$. Given $f$ a local solution to \eqref{keyODE} and $s_o\in I$, such that $(f'(s_o))^2<1$, then $f$ can be globally extended to $f:I\to\R$. In particular, the associated  $\Sigma$-invariant translating soliton is graphical and it can be globally defined on $M$. 
\end{corollary}
\begin{proof} By Theorem \ref{equiv_ODE}, any $\Sigma$-invariant translating soliton will be constructed by a solution $f$ to \eqref{keyODE}. We make the change $w=f'$, so it reduces to $w'(s) = (\tilde{\ep}+\ep'w(s)^2) (1 - w(s)h(s))$. Since $\tilde{\ep}\ep'=-1$, the constant functions $w(s)=\pm 1$ are solutions to this differential equation. Then, given an initial condition $(s_o,f_1)\in I\times (-1,1)$, there exists a local solution $w:(s_o-\rho,s_o+\rho)\rightarrow\R$ such that $w(s_o)=f_1$ and $\vert w\vert<1$. By the uniqueness of solutions to IVP, $w$ cannot reach the values $\pm 1$, and so, it can be globally extended to $w:I\to(-1,1)$. We define  $f(s)=f_o+\int_{s_o}^s w(x)dx$ for some $f_o\in\R$. By using the curve $\delta:I\to I\times \R$, $\delta(s)=(s,f(s))$, we can (re)construct our $\Sigma$-invariant translating soliton by \eqref{resureccion}.
\end{proof}
\begin{remark} This section generalizes several classical results in a general geometric context. A well-known example is the case of a rotationally invariant translating soliton in $\mathbb{R}^{n+1}$. We consider the Lie group $SO(n)$ acting by isometries on $\mathbb{R}^n$ which gives rise to the Riemannian submersion $\pi:\mathbb{R}^{n}\backslash\{0\}\rightarrow\mathbb{R},\,\pi(x)=\|x\|$\color{black}, and the Riemannian product $\bar{M}=(\mathbb{R}^n\times\mathbb{R},g^{\mathbb{R}^n}+dt^2)$. An easy computation shows that the mean curvature of the fiber $\pi^{-1}(s)$ is $h(s)=(n-1)/s$. Using Theorem \ref{equiv_ODE}, we recover the well-known ODE for rotationally invariant translating solitons $f''=(1+f^2)(1-\frac{n-1}{s}f')$, studied in \cite{CSS}, where such solitons are classified. In that paper, it is shown that there exist only two types: a globally defined graphical soliton, the so-called bowl soliton or  translating paraboloid,  which can be seen as a particular application of our Corollary \ref{global-solutions}, and a family of non-graphical wing-like translators or translating catenoids, obtained by gluing two graphical solutions, which are a special case of Corollary \ref{oide}. 
$\Box$ 
\end{remark}

\section{All solutions to an ODE}\label{soluODE}

In this section, we will need some tools which can be found in the book \cite{Wiggins}. Our targent is to find all local solutions to the following equation, for $s>0$, 
\begin{equation} \label{yasesabe}
f''(s)=\big(1- f'(s)^2\big)\Big(1-\frac{n-1}{s}f'(s)\Big).
\end{equation}
or rather
\begin{equation}\label{unordenmenos}
w=f', \quad w'=\big(1- w^2\big)\Big(1-\frac{n-1}{s}w\Big).
\end{equation}
where $n\geq 2$ is a natural number. Clearly, there are two degenerate examples, namely,  $\hat{f}_{\pm}:(0,\infty)\rightarrow\R$, 
$\hat{f}_{\pm}(s)=\pm s+f_o$, $f_o\in\R$, that is, $\hat{w}_{\pm}(s)=\pm 1$.  

\subsection{The case when $\vert x\vert <1$}

The following particular example was studied in \cite{JLC}. 
\begin{lemma}\label{bowl} There exists a unique function $f_B:\mathbb {R}\to\mathbb{R}$ such that $f$ is even, analytical, and such that its derivative $w_B=f'$ is the unique solution to the boundary problem
\[ w'=\big(1- w^2\big)\Big(1-\frac{n-1}{s}w\Big), \quad w(0)=0.\]
\end{lemma}

We  consider the following open domains
\begin{gather*} \Omega = \{(s,x)\in \R^2 : s>0, \ \vert x\vert <1\}, \\ 
\Omega_1=\{(s,x)\in\Omega : w_B(s)>x\}, \quad  \Omega_2 =\{(s,x)\in\Omega : w_B(s)<x\}. 
\end{gather*}
\begin{lemma} \label{moduloxmenor1} 
Any unextendable solution $w$ inside $\Omega$ is globally defined, $w:(0,\infty)\to\R$, and it is one of the following:
\begin{description}
\item[Case I:] The solution $w_B=f':[0,\infty)\to \R$ of Lemma \ref{bowl};
\item[Case II:] $w>w_B$ everywhere, with one critical point $s_1>0$, such that $w(s_1)=s_1/(n-1)$, and $\lim_{s\to 0}w(s)=1$.
\item[Case III:] $w<w_B$ everywhere, without critical points, strictly increasing, with no critical points, and $\lim_{s\to 0}w(s)=-1$.
\end{description}
In addition, all of them (also $w_B$) satisfy $\lim_{s\to \infty}w(s)=1$.
\end{lemma}
\begin{proof} Take $(s_o,w_o)\in\Omega$  initial conditions for \eqref{unordenmenos}, $w(s_o)=w_o$, and let $w$ be the solution. By the solutions $\hat{w}_{\pm}$, it will be globally defined  $w:(0,\infty)\rightarrow\R$ and its graph will remain inside $\Omega$. Then, $w'(s)=0$ for some point $s\in (0,\infty)$ if, and only if, either $w(s)=\pm 1$ (excluded) or $w(s)=s/(n-1)$, at some point $s\in(0,\infty)$. In other words, $w$ will admit a critical point if its graph intersects the line $r$ of equation $x=s/(n-1)$ horizontally. As the line $r$ is a monotonic curve, the solution can only intersects it at most once due to our ODE. 

\noindent \textbf{Case I:} We already know the solution  $w_B=f':[0,\infty)\to \R$ of Lemma \ref{bowl}. 

\noindent \textbf{Case II:} Take initial conditions $(s_o,w_o)\in\Omega_1$. The associated solution $w:(0,\infty)\rightarrow \R$ will remain in $\Omega_1$, due to the uniqueness of solutions to IVP. As  $w(s)<w_B(s)<s/(n-1)$ for any $s>0$, then $w$ has no critical points, and $w'(s_o)>0$, 
so it is strictly increasing for any $s>0$. 

Since, $w(s)<1$ for any $s>0$, there exists $\lim_{s\to +\infty} w(s)=w_1\in (-1,1]$ and $\lim_{s\to +\infty}w'(s)=0$. However, $0=\lim_{s\to +\infty}w'(s)=\lim_{s\to +\infty} (1-w(s)^2)(1-(n-1)w(s)/s)$  $=1-w_1^2$. Therefore, $\lim_{s\to +\infty} w(s)=1.$

As  $w_B(0)=0$, and $w(s)<w_B(s)$, there exists $s_1>0$ such that $w(s_1)=0$. We want to compute the limit of $w(s)$ and of $w'(s)$ when $s\to 0$. As $-1<w(s)$ and $w$ is strictly increasing, then $\lim_{s\to 0}w(s)=w_2\in [-1,0)$. Define now $z:\R\rightarrow\R$, $z(t)=w(e^t)$. Then, $z$ is still strictly increasing, and $z'(t)=(1-z(t)^2)(e^t-(n-1)z(t))$. Clearly, $w_2=\lim_{t\to -\infty}z(t)$ and this implies $0=\lim_{t\to -\infty}z'(t)=\lim_{t\to -\infty} (1-z(t)^2)(e^t-(n-1)z(t)) = -(n-1)(1-w_2^2)w_2$. This means that $w_2=-1$. To summarize, 
\[ \lim_{s\to 0}w(s)=-1, \quad \lim_{s\to 0}w'(s)=0.\]

Again, by the uniqueness of solutions to IVP, we can parametrize this family by considering the initial values, $(0,+\infty)\ni s_0$, that is $(s_0,0)\in\Omega_1$, and the union of all the graphs of the solutions will foliate $\Omega_1$.\\

\noindent \textbf{Case III:} Take initial conditions $(s_o,w_o)\in\Omega_2$. The associated solution $w:(0,\infty)\rightarrow \R$ will remain in $\Omega_2$, similarly to Case II. If $w'(s)\neq 0$  for all $s$, its graph cannot cross the line $r$ of equation $x=s/(n-1)$ and it is strictly above the graph of $w_B=f_B'$. But then, either $w(s)>s/(n-1)$ for any $s>0$, which is impossible since $w$ is bounded by $1$ and defined for all $s>0$, and the line $r$ crosses the horizontal line $x=+1$; or 
$s/(n-1)>w(s)>w_B(s)$ for any $s>0$, which shows that $\lim_{s\to 0}w(s)=0$. Then, we found a second solution to the  problem \eqref{unordenmenos} with boundary value $w(0)=0$. By the uniqueness of the solution to this boundary problem, we get to a contradiction. Thus, there exists a unique $s_1>0$ such that $w(s_1)=s_1/(n-1)$, which is the only critical point. In this way, we  parametrize this family by the open segment $r\cap\Omega_2$. It is easy to compute that 
\[ w''(s_1) = \frac{-1}{s_1}\left(1-\frac{s_1^2}{(n-1)^2}\right)<0.
\]
Therefore, $s_1$ is a absolute minimum of $w$. This means that $w$ is strictly decreasing when $0<s<s_1$ and strictly increasing when $s>s_1$. For big $s>0$, we know $w_B(s)<w(s)<1$, which shows $\lim_{s\to +\infty}w(s)=1.$ By similar computations to Case II, $\lim_{s\to 0}w(s)=1$. 
\end{proof}

The following picture illustrates the solutions to \eqref{unordenmenos} such that  $-1<w=f'<1$, for $n=3$, made with wxMaxima\textcopyright. 
\begin{center}
\includegraphics[scale=0.35]{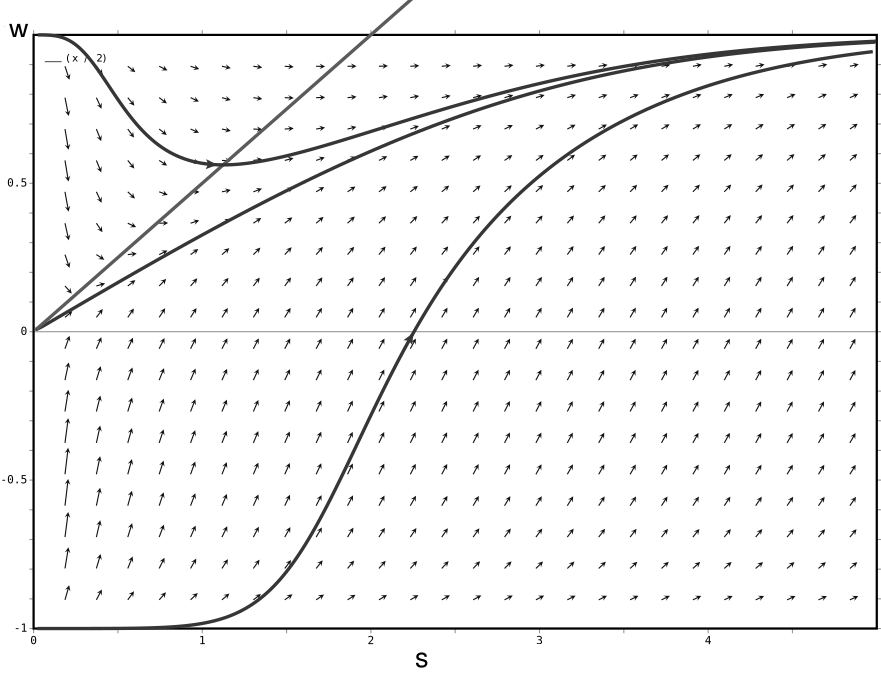}
\end{center}

\subsection{The case when $\vert x\vert >1$}

For the following study, we define the open domains
\[ \Gamma_{-} = \{(s,x)\in\R^2 : s>0, \ x<-1\}, \quad 
\Gamma_{+} = \{(s,x)\in\R^2 : s>0, \ x>1\}.\]
\begin{lemma} {\normalfont \textbf{(Case IV)}} \label{caseiv} Any unextendable solution $w$ contained in $\Gamma_{-}$ is defined on $w:(0,s_1)\to\R$, for some $s_1>0$, it is strictly increasing, with $\lim_{s\to 0}w(s)=1$ and having a finite time blow-up at $s_1$. 
\end{lemma}
\begin{proof}
Let $(s_0,w_o)\in\Gamma_{-}$ initial solutions for \eqref{unordenmenos}, $w(s_o)=w_o<-1$. A local solution $w:(s_o-\delta,s_o+\delta)\to\mathbb{R}$ stays inside of  $\Gamma_{-}$. Since the line $r$ does not intersects $\Gamma_{-}$, $w$ cannot have any critical point. And due to  
\[ w'(s_o) = (1-w_o^2)\left(1 -\frac{(n-1)w_o}{s_o}\right)<0,
\]
then $w$ is always strictly decreasing. As it is bounded by the constant solution $\hat{w}_{-1}(s)=-1$, $w$ can be extended to $(0,s_o+\delta)$. Also, a simple computation shows $\lim_{s\to 0}w(s)=-1$. 
Next, since $1-(n-1)w_o/s_o>1$, for any $s\geq s_o$ we have,
\[ (1-w(s)^2)\left(1-\frac{(n-1)w(s)}{s}\right) \leq 1-w(s)^2<0. 
\]
Therefore, the ODE \eqref{unordenmenos} is dominated by the ODE $z'(s)=1-z(s)^2$, whose general solution is 
$z(s)=\coth(s+\rho)=\frac{e^{2(s+\rho)}+1}{e^{2(s+\rho)}-1}$, $\rho\in\R$. We choose $\rho$ such that $z(s_o)=w_o=w(s_o)<-1$. From here, $e^{2(s_o+\rho)}-1<0$ and  $\rho<0$, so that $z$ can be extended at most to $(0,-\rho)$. But as $w'(s)\leq z'(s)<0$, then $w$ has a finite-time blow-up at some $s_1<-\rho$.
\end{proof}
\newcommand{\w}{\mathbf{w}}
\begin{lemma}\label{clasificacion-timelike} Any unextandable solution to \eqref{unordenmenos} included in $\Gamma_+$ is one of the following:
\begin{description}
\item[Case V:] There exists a unique $\w:(0,+\infty)\to\R$ which is  strictly increasing, asymptotic to the line $r$ of equation $x=s/(n-1)$ at infinity, and $\w(s)>s/(n-1)$ for any $s>0$.
\item[Case VI:]  $w:(0,+\infty)\to\mathbb{R}$ is globally defined, with just one critical point $s_1>s_o$, $\lim_{s\to +\infty}w(s)=1$, $w(s)<\w(s)$ for any $s>0$.
\item[Case VII:] $w:(0,s_1)\to\mathbb{R}$ has a finite time blow up at certain $s_1>s_o$, without critical points, and $w(s)>\w(s)$ for any $s<s_1$.
\end{description}
In addition, all of them satisfy (also $\w$), $\lim_{s\to 0}w(s)=1$ and $\lim_{s\to 0}w'(s)=0$.
\end{lemma}
\begin{proof}
Take $(s_o,w_o)\in\Gamma_{+}$  initial conditions for \eqref{unordenmenos},  and let $w$ be the solution such that $w(s_o)=w_o$. Its graph will remain inside $\Gamma_{+}$. Also, $w$ will have a critical point if its graph intersects the line $r$ horizontally ($w'(s_1)=0$ for some $s_1>0$), i.~e, $w(s_1)=s_1/(n-1)$.  For any point $s>s_1$, $w'(s)<0$, whereas for any $s<s_1$,  $w'(s)>0$ holds. In such case, $s_1$ will be the only critical point, and its absolute maximum. 
\newcounter{step}
\newcommand{\step}{\refstepcounter{step}
\par\noindent\textsc{Step \thestep:}\, }

\step Consider $(s_o,w_o)\in \Gamma_{+}$ such that $w_o=s_o/(n-1)$, and a solution to \eqref{unordenmenos}  around $s_o$, namely $w:(s_o-\delta,s_o+\delta)\rightarrow\R$, with $w(s_o)=w_o=s_o/(n-1)$. By \eqref{unordenmenos}, $w_o'(s_0)=0$, and it is its unique critical point. 

Then, $w$ will always be decreasing for $s>s_0$. 
By the barrier solution $w_1(s)=+1$, $w$ can be extended to $w:(s_o-\delta,+\infty)\rightarrow\R$, with $\lim_{s\to +\infty}w(s)=w_1\geq 1$, so that $0=\lim_{s\to +\infty}w'(s)=\lim_{s\to +\infty}(1-w(s)^2)\big(1-(n-1)w(s)/s\big)=1-w_1^2$. This means
$$\lim_{s\to +\infty}w(s)=+1.$$
On the other hand, $w$ will always be increasing for $s<s_o$. Hence we can extend $w:(0,+\infty)\to\R$, $w$ increasing on $(0,s_o)$, and $w(0)=w_2\geq 1$. As in Case II, by the change of variable $s=e^t$, we get immediately that 
$\lim_{s\to 0}w(s)=+1.$

\step As in Corollary \ref{oide}, we get for each $s_0>0$ a solution $w=f_-'$ such that $\lim_{s\to s_0}w(s)=+\infty$, 
i.e. with finite time blow up in $s_0$. As in step 1, $w$ can be extended to $(0,s_0)$ and  similarly 
$\lim_{s\to 0}w(s)=+1.$

\step We want to obtain the first type solution as explained above. 
We fix now $s_o=n-1$. Define 
\[J=\{w_o\in \R: w_o>1, \ \exists w:(0,+\infty)\rightarrow\R, \ \mathrm{solution\ to}\ \eqref{unordenmenos},  \ w(n-1)=w_o\}.\]
By step 1, $J\neq\emptyset$. By step 2 we have a solution $w_n:(0,n)\to \mathbb{R}$, such that $\lim_{s\to n}w_n(s)=+\infty$. Then,  we have that $j\leq w_n(n-1)$ for all $j\in J$. Thus, we can define $A:=\sup J<w_n(n-1)$. 

Now define $\w :(0,n-1+\delta)\to\mathbb{R}$ the solution to \eqref{unordenmenos} such that  $\w (n-1)=A$. Assume that $\lim_{s\to s_1} \w =+\infty$ for some $s_1>n-1$. But then let $s_2=s_1+\varpi$, $\varpi>0$. By step 2, there exists a blow up solution $w_2$, such that $\lim_{s\to s_2}w_2(s)=\infty$. Therefore $\w >w_2$ for all $s\in (0,s_1]$. This is a contradiction with $A=\sup J$. This means that $\w $ is globally defined.

Moreover, $\w $ has no critical points. Indeed, assume that there is a $s_c>0$, such that $\w '(s_c)=0$. As before , $(s_c,w(s_c))\in r$. Then,  by step 1 there exists another globally defined solution $w$, such that $w'(s_c+1)=0$, but then by uniqueness of solutions $w>\w $ for all s, which is a contraction to $A=\sup J$.  Therefore $\w $ is strictly increasing and bounded below by the line $r$, so that 
$$\lim_{s\to +\infty}\w (s)=+\infty.$$

We are going to show that $\w $ is asymptotic to the line $r$. Let $K>1$. Suppose that for large enough $s_1>0$, $\w '(s)>K/(n-1)$ for any $s\geq s_1$. Then, there exists $c\in \mathbb{R}$ such that for all $s\geq s_1$, $\w (s)>Ks/(n-1)+c$. But now, there exists $c_o>0$ such that $\frac{(n-1)\w (s)}{s}-1 > \frac{n-1}{s}\left(\frac{Ks}{n-1}+c\right)-1=K-1+\frac{(n-1)c}{s}>c_0>0$ for large enough $s\geq  s_2\geq s_1$. And so, $\w '(s)>c_o\left(\w (s)^2-1\right)$ for any $s\geq s_2$.  Then, \eqref{unordenmenos} dominates the ODE $g'=c_o(g^2-1)$. The solution to this equation such that $g(s_2)=w_o$ is given by 
\[ g(s)= \frac{e^{-2c_o(s-s_2)}+a}{e^{-2c_o(s-s_2)}-a}, \quad a=\frac{w_o-1}{w_o+1}<1. 
\]
Since $0<a<1$, $g$ is well defined and positive. Therefore, for any $s\geq s_2$, it holds $\w (s)\geq g(s)$. However, $g$ has a blow up at $s_3=s_2-\ln(a)/(2c_o)$. This shows that $\w $ has a blow up at $s_4<s_3$. This is a contradiction.  Therefore, since $\w $ cannot cross $r$, we get using  L'Hospital that 
\[\lim_{s\to +\infty}\w '(s)=\frac{1}{n-1}.\] 
Now we notice that since $\w >1$, \eqref{unordenmenos} is equivalent to the expression
$$(n-1)\w (s)-s=\frac{s\w '(s)}{\w (s)^2-1}.$$
But now, as $\w $ is defined until $+\infty$, and without critical points,  
\begin{align*}
& \lim_{s\to +\infty} [(n-1)\w (s)-s] = \lim_{s\to +\infty} \frac{s\,\w '(s)}{\w (s)^2-1} =\frac{1}{n-1} \lim_{s\to \infty} \frac{s}{\w (s)^2-1} \\
& = \frac{1}{n-1}\lim_{s\to +\infty} \frac{1}{2\w (s)\w '(s)}=0.
\end{align*}
Summing up, the solution $\w $ is asymptotic to the line $r$. 

Since $A=\sup J$, take any other globally defined solution $w_B:(0,+\infty)\to \mathbb{R}$ such that $w_B(n-1)<\w (n-1)$ and $w_B<\w $ everywhere. Assume that $w_B$ has no critical points. By similar computations as above, $w_B$ will also be asymptotic to the line $r$. Then, the difference $h=\w -w_B$ will satisfy $h>0$ and $\lim_{s\to +\infty}h(s)=0$. By \eqref{unordenmenos}, $\w '>w_B'$ everywhere. But then, $h'>0$, that is to say, $h$ is strictly increasing everywhere. This is a contradition with $\lim_{s\to +\infty}h(s)=0$. 

In other words, there is a unique solution $\w $ globally defined on $(0,+\infty)$, with no critical points. We will keep this notation for the rest of the paper. 

Finally, given $(s_o,w_o)\in \Gamma_{+}$, such that $\w (s_o)<w_o$, consider the solution to \eqref{unordenmenos} such that $w(s_o)=w_0$. As before, there exists $w:(0,s_o+\delta)\to\R$ for some (small) $\delta$. Assume that $w$ is globally defined on $(0,+\infty)$. Then, $w(n-1)>\w (n-1)=A$. This is a contradiction. Therefore, any other solution over $\w $ admits a finite-time blow-up. 
\end{proof}
The following picture illustrates this lemma, for $n=3$, made with wxMaxima\textcopyright. 
\begin{center}
\includegraphics[scale=0.35]{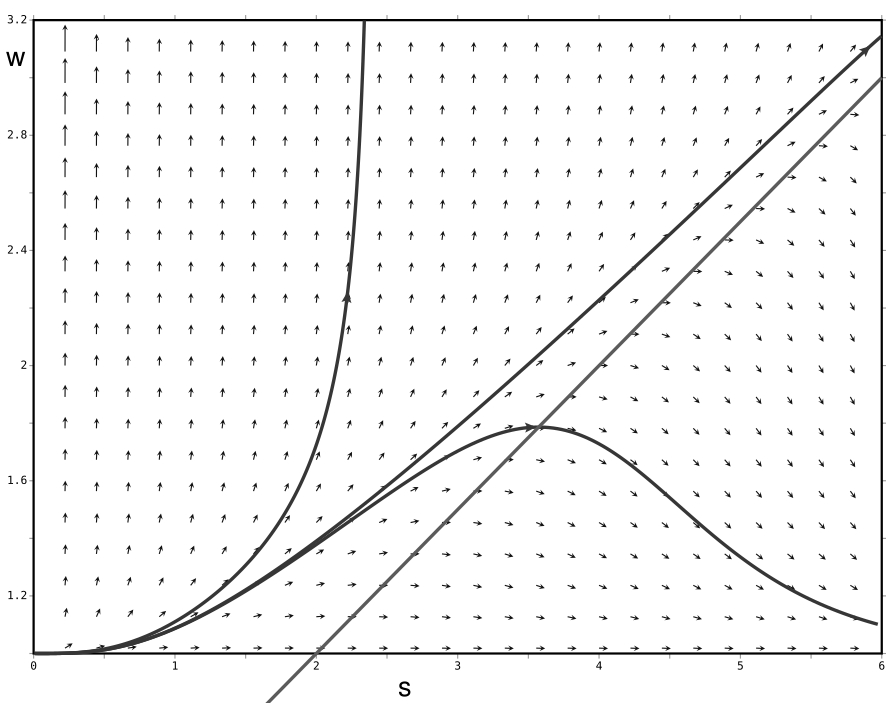}
\end{center}

\section{The Action of the Group $SO(n)$}
\label{Rotational}

We consider $\R^n$ with usual coordinates $x=(x_1,\ldots,x_n)$. To obtain  a smooth Riemannian submersion, we restrict it to $\pi:M=\R^n\backslash\{0\}\to (0,+\infty)$, $\pi(x)=\sqrt{\sum_i x_i^2}$. Then,  
\[ \grad \pi (x)= \sum_i \frac{x_i}{\pi(x)}\partial_i\vert_x, \quad \vert \grad\pi(x)\vert^2 = \sum_i\frac{x_i^2}{\pi(x)^2} = 1, \quad \mathrm{div}(\grad \pi)(x) = \frac{n-1}{\pi(x)},
\]
for any $x\in M$. Note that our map $\pi$ is invariant by the standard action of $SO(n)$ on $\R^n$. Consider now the Minkowski space $\mathbb{L}^{n+1}=\R^n\times\R\ni (x_1,\ldots,x_{n+1})$ with standard flat metric $g=\sum_{i=1}^ndx_i^2-dx_{n+1}^2$. In order to obtain a graphical $SO(n)$-invariant translating soliton, we consider $u=f\circ\pi:M\rightarrow\R$, where $f:I\subset (0,+\infty)\rightarrow\R$ is the desired function, which will be a solution to the ODE \eqref{keyODE}, which turns out to be 
\begin{equation} \label{yasesabe2}
f''(s)=\big(1- f'(s)^2\big)\Big(1-\frac{n-1}{s}f'(s)\Big).
\end{equation}
Then, we resort to Section \ref{soluODE}, and make a geometrical interpretation. Clearly, the submersion $\pi$ is Riemannian, hence $\tilde{\ep}=1$; 
the product is Lorentzian, and therefore $\ep'=-1$. Moreover, $\ep = \mathrm{sign}(-1+\vert\grad u\vert^2) = \mathrm{sign}(-1+(f'(\pi))^2)=\pm 1$.  If $\ep=-1$,  the rotational surface is spacelike, whereas $\ep=+1$ gives rise to a timelike surface. 
The spacelike examples are characterized by $(f')^2<1$, since $\varepsilon=-1$, but the timelike examples are those coming from $(f')^2>1$. 
\begin{examp}\normalfont From the above section, from each (inextendible) solution $w$ to \eqref{keyODE}, we consider a primitive $f=\int w$, which will be a solution to \eqref{yasesabe}. Then, the hypersurface 
\[ \Sigma_f:\Omega\subset \R^n\to \L^{n+1}, \quad \phi(x)=(x,f(\pi(x)))
\]
is a $SO(n)$-invariant translating soliton. We are extending continously $\pi$ to the rotation axis such that $\pi(x)=0$ when $x$ is a point of the axis. By the limits of $w$ at $0$ (Lemmatta \ref{caseiv} and \ref{clasificacion-timelike}), all these solutions $f$ are asymptotic at the axis to a light-like straight line of equation $x=\pm s+s_2$. Thus, all $\Sigma_{f}$ will be asymptotic at the axis of rotation to a light-like cone.  
\begin{itemize}
\item \textit{The Ogival Paraboloid:} $\Omega=\R^n$. Case V in Lemma \ref{clasificacion-timelike}, the function  $f=\int \w$ is asymptotic at infinity to an \textit{Euclidean parabola}
\item \textit{Timelike Calyx:} $\Omega=\R^n$. Case VI of Lemma \ref{clasificacion-timelike},  since $w$ satisfies $1=\lim_{s\to 0}w(s)=\lim_{s\to \infty}w(s)$, then $\Sigma_{f}$ will be asymptotic to two upper light-like cones, one at the origin and one at infinity. 
\end{itemize}
\end{examp}
\begin{examp} \textit{The spindle:} \normalfont We consider an example of Lemma \ref{caseiv}, Case IV, $w$ and $f_{-}=\int w$.  Since $w<-1$, then $f_{-}$ is injective, so we can consider its inverse $\alpha=f_{-}^{-1}$. Both functions are strictly decreasing. As $f_{-}$ is defined on $(0,s_o)$, there exists $\alpha:(a,b)\to\R$ such that $\lim_{y\to a}\alpha(y)=s_o$ and $\lim_{y\to b}\alpha(y)=0$.  We are going to use the proof of Corollary \ref{oide}. Then, $\alpha$ is a solution to the ODE
\[
\alpha''(y) = \left(-1 + \alpha'(y)^2\right) \left( \frac{n-1}{\alpha(y)} - \alpha'(y)\right).
\]
Assume  $a=-\infty$. Since $\lim_{s\to s_o}w'(s)=-\infty$, then $\lim_{y\to -\infty}\alpha'(y)=0$, so that it also holds $\lim_{y\to -\infty}\alpha''(y)=0$. By inserting this in the ODE, $0=\lim_{y\to -\infty} \alpha''(y)=(1-n)/s_o$. This is a contradiction. 

Therefore, $a\in\R$. Thus, we can extend $\alpha$ a little, $\alpha:(a-\epsilon,b)\to\R$. But at $a$, $\alpha'(a)=0$, and $\alpha(a)=s_1\in(0,s_o)$. By the proof of Corollary \ref{oide}, $\alpha$ can be split in two inverse functions, one after $a$ (which is the original $f_{-}$) and one before $a$, which we call $f_{+}$. Now, it is quite clear that $w_{+}={f}_{+}'>1$ provides an example of Case VII in Lemma \ref{clasificacion-timelike}. This means that the union of these two examples make a new type of translating soliton. This hypersurface is a (topological) sphere with two lightlike-conic  singularities, namely, asymptotic to two different light-like cones, one in the upper half and one in the lower half, with both vertices on the rotation axis. We will call it a \textit{fusiform hypersurface} or a \textit{spindle}. 
A similar symmetric reasoning holds if we start with an example of Case VII in Lemma \ref{clasificacion-timelike}, arriving to an example of Lemma \ref{caseiv}. 
\end{examp}

\begin{theorem}\label{timelike} Up to isometries, any  timelike, $SO(n)$-invariant translating soliton in $\L^{n+1}$, $n\geq 2$, is an open subset of either an Ogival Paraboloid, or a Timelike Calyx, or a Spindle. 	
\end{theorem} 

The wing-like technique used in the previous theorem makes no sense for spacelike hypersurfaces, because at some point, the hypersurfase would be parallel to the timelike rotation axis. 

\begin{remark}
In \cite{coreano}, independently, the author obtained the \textbf{spacelike} examples, obtaining three types. In Proposition 3.2, the profile curve of type 2 gives rise to the \textit{Bowl}. However, from our point of view, the geometrical reconstruction of the two new surfaces can be slightly improved. Indeed, a profile curve of type 1 in Proposition 3.2 will hit the axis by a $-\pi/4$ angle. However, if one tries to extend it beyond the axis, by rotating the curve, we must arrive again to one of the three cases of Proposition 3.2. But now, the angle will be $\pi/4$, so this extension has to be a profile curve of type 3. This reasoning can be reversed, starting from a profile curve of type 3, and ending up with a profile curve of type 1. Moreover, the \textit{timelike} examples were not considered in this paper. 
\end{remark}

\begin{theorem}\label{spacelike} Up to isometries, any $SO(n)$-invariant, spacelike, translating soliton in $\L^{n+1}$, $n\geq 2$, is an open subset of a hypersurface generated by the graph of any of the three cases in Lemma \ref{moduloxmenor1}.
\end{theorem}
Similarly to the Euclidean case, we  call the hypersurface generated by case I, the \textit{Bowl}.

\begin{corollary} The bowl is the only entire, $SO(n)$-invariant, graphical translating soliton  in $\L^{n+1}$, $n\geq 2$, up to isometries.
\end{corollary}

\begin{remark}\normalfont \textit{Timelike translating solitons do not satisfy a tangency principle.} Indeed, we consider two different spindles. At the points where the distance to the axis is maximal, the tangent planes contain directions parallel to the rotation axis. Due to the rotation action of $SO(n-1)$, the study of the relative position of the hypersurfaces will reduce to the study of the curves from Corollary \ref{oide}. We recall that the differential equation is
\[
\alpha''(y) = \left(-1 + \alpha'(y)^2\right) \left( \frac{n-1}{\alpha(y)} - \alpha'(y)\right), 
\]
with initial conditions $\alpha_i(y_0)=s_i$, $\alpha'(y_0)=0$, for $i=1,2$. We can assume $0<s_1<s_2$. Then, 
\[ \alpha_1''(y_o)=\frac{1-n}{s_1}<\frac{1-n}{s_2}=\alpha_2''(y_o)<0.
\]
To compare their graphs, we use $\beta=\alpha_1+\alpha_2(y_o)-\alpha_1(y_o)$. Clearly, $\beta'=\alpha_1'$ and $\beta''=\alpha_1''$. But now, $\beta(y_0)=\alpha_2(y_o)$, $\beta'(y_o)=0=\alpha_2'(y_o)$ and $\beta''(y_o)<\alpha_2''(y_o)$. Therefore,  $\beta(y)<\alpha_2(y)$ holds for any $y\neq y_o$  in a neighborhood of $y_o$. Geometrically, this is the same as moving one of the spindles in such a way that  their tangent planes coincide (parallel to the rotation axis), but one spindle is at one side of the other spindle. $\Box$   
\end{remark}

The following pictures were made with wxMaxima\copyright\ and  Gnuplot\copyright.
\begin{center}
\begin{longtabu}{ccc}  
\multicolumn{3}{c}{\textbf{Spacelike examples}} \\ \hline 
 && \\
	\includegraphics[height=3.7truecm]{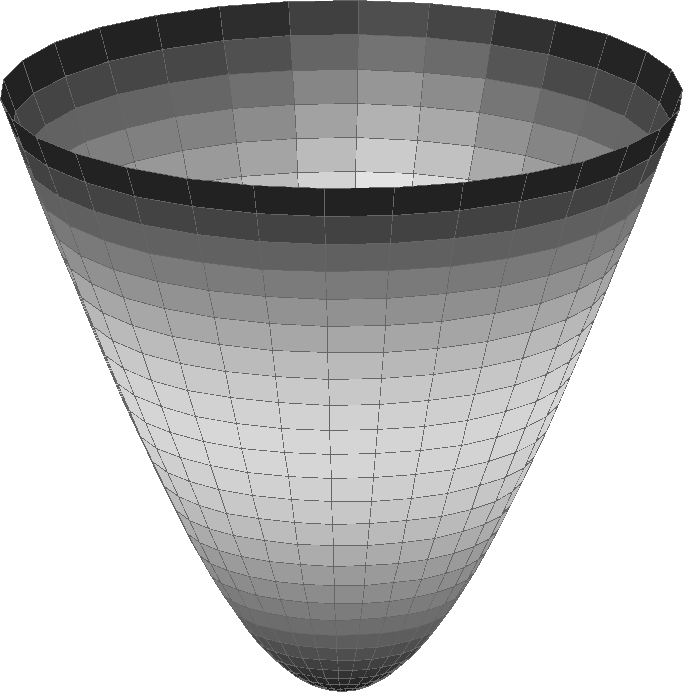} &
	\includegraphics[height=3.7truecm]{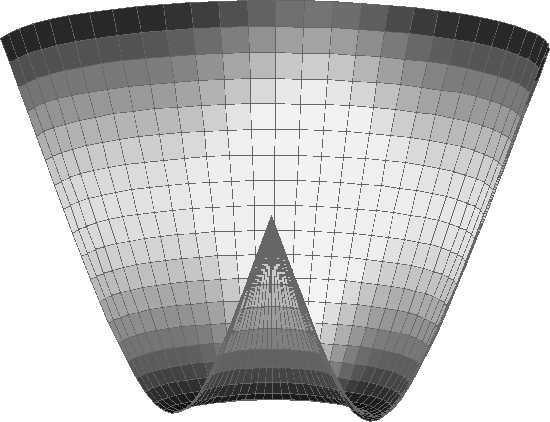}& 
	\includegraphics[height=3.7truecm]{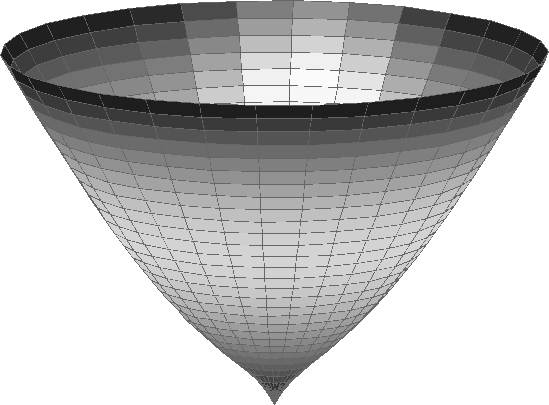} 
\\  
	The bowl (Case I.) & Case II.  
	& Case III.  
\\   
&&\\ 
\multicolumn{3}{c}{\textbf{Timelike examples}} \\ \hline 
&& \\
		\includegraphics[height=3.7truecm]{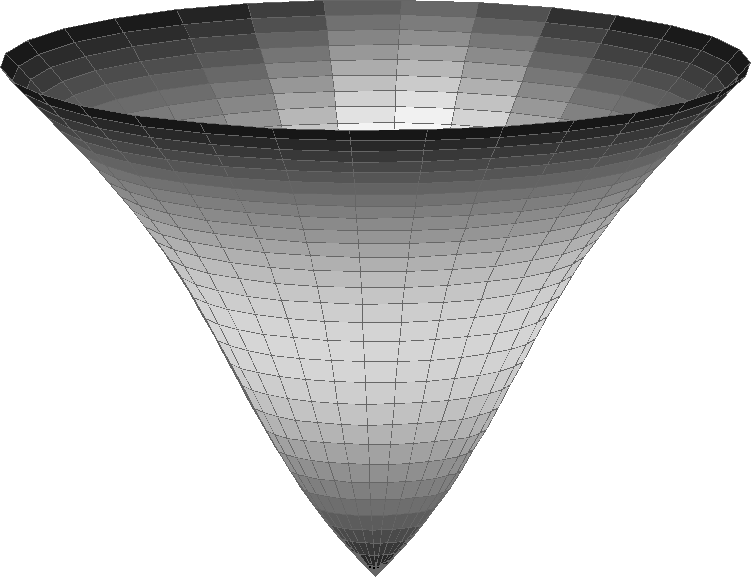}&
		\includegraphics[height=3.7truecm]{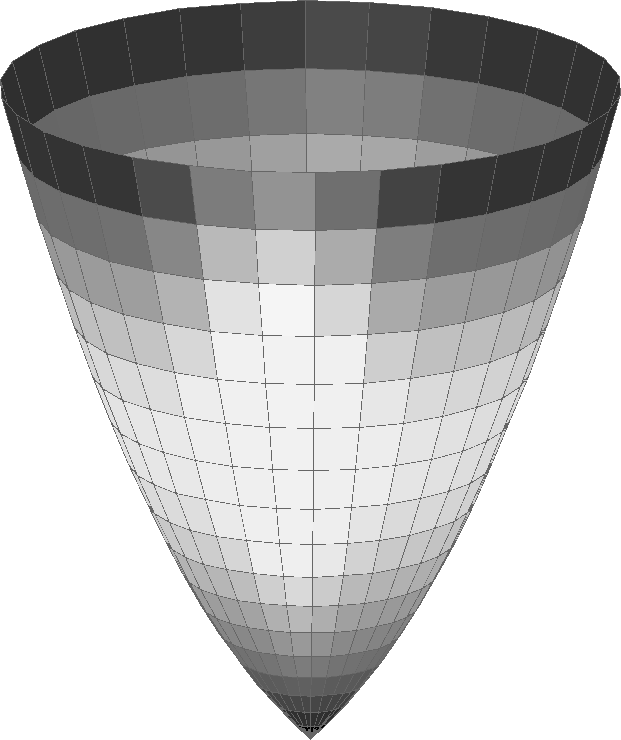}&
		\includegraphics[height=3.7truecm]{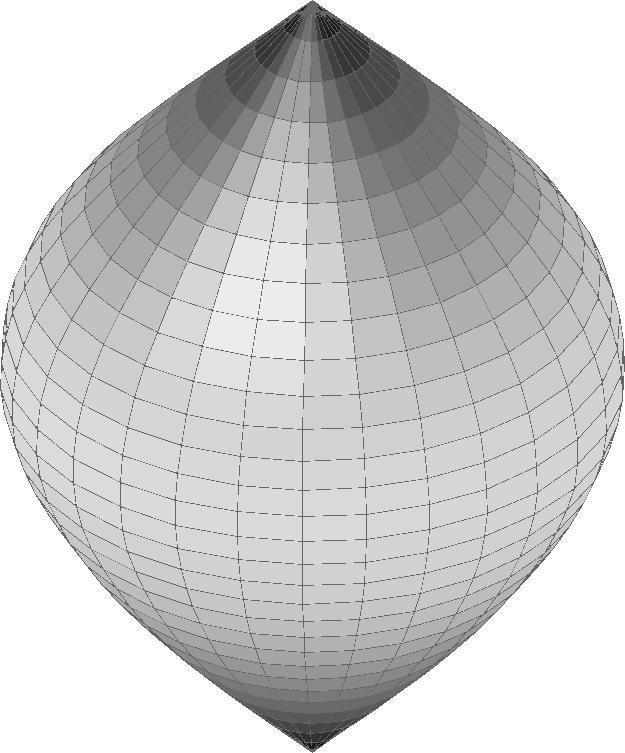} \\  
		A Timelike Calyx & The Ogival Paraboloid & A Spindle \\  \hline 
\end{longtabu}
\end{center}

\section{The action of the group $\SOf$}\label{SOf}

Let $\L^{n}$, $n\geq 2$, be the Minkowski (linear) space with the standard metric $g(X,Y)=X_1Y_1+\ldots+X_{n-1}Y_{n-1}-X_nY_n$, $X=(X_1,\ldots,X_n)^t, Y=(Y_1,\ldots,Y_n)^t\in\L^n$. We consider the  important  subgroup of linear isometries
\begin{align*}
 SO^{\uparrow}(n-1,1)  =& \left\{ 
A \in \mathcal{M}_n(\R) : \det(A)=+1, \ A^tI_{n-1,1}A=I_{n-1}, \right. \\
& \quad \left. A \mbox{ preserves the time orientation}\right\},
\end{align*}
where $I_{n-1,1}=\mathrm{Diagonal}(1,\ldots,1,-1)$, $A^t$ is the transpose of $A$. The \textit{light cone} $\mathcal{C}=\{p\in\L^n: g(p,p)=0\}$ can be split into the zero, the future part and the past part, $\mathcal{C}^{\uparrow}=\{p\in\mathcal{C} : p_n>0\}$, $\mathcal{C}^{\downarrow}=\{p\in\mathcal{C} : p_n<0\}.$ We  call $\mathbf{T}^{\uparrow}$  the open subset of all future pointing time-like vectors, whereas $\mathbf{T}^{\downarrow}$ is the open subset of all past pointing time-like vectors.  Let $\mathcal{S}$ be the open subset of all space-like vectors, which is connected when $n\geq 3$, and it has two connected components when $n=2$. Then, we have the disjoint union $\L^n = \{0\}\cup\mathcal{C}^{\uparrow}\cup \mathcal{C}^{\downarrow}\cup \mathbf{T}^{\uparrow}\cup \mathbf{T}^{\downarrow}\cup \mathcal{S}$. Each of them are invariant by $SO^{\uparrow}(n-1,1)$.  

Firstly, we study the case $n=2$. Take the Minkowski plane $\L^2$ with standard flat metric $g=dx^2-dy^2$. The \textit{Boost} Lie group is the set of matrices
\[ 
SO^{\uparrow}(1,1)=\mathbf{B} = \left\{ A_{\theta} = \begin{pmatrix} \cosh(\theta) & \sinh(\theta) \\
\sinh(\theta) & \cosh(\theta) \end{pmatrix} : \theta\in\R 
\right\}, 
\]
which acts on $\L^2$ by isometries. If we remove the origin $(0,0)$, the action is free, but it is not proper, because the quotient space $(\L^2\backslash\{0\})/\mathbf{B}$ is not Haussdorf. Indeed, the light cone projects to four points if $n=1$  which cannot be separated by neighbourhoods. We split the plane in four regions whose boundaries are made of two light-like geodesics, namely
\begin{align*}
\Omega_1 =  \{ (x,y)\in\L^2 : y^2<x^2, \ 0<x \}, & \quad 
\Omega_2 =  \{ (x,y)\in\L^2 : y^2>x^2, \ 0<y \}, \\
\Omega_3 =  \{ (x,y)\in\L^2 : y^2<x^2, \ 0>x \}, &  \quad 
\Omega_4 =  \{ (x,y)\in\L^2 : y^2>x^2, \ 0>y \}.
\end{align*}
We will use the globally defined orthonormal frame $\{\partial_x,\partial_y\}$. 
The action of $\mathbf{B}$  works very well on each $\Omega_k$, $k=1,2,3,4$, because they admit decompositions as in \eqref{equ:decom} (see below). According to \cite{BarrosCaballeroOrtega}, first we will obtain the \textit{fundamental} examples  included in the fundamental regions $\Omega_k\times\R\subset\L^3$, $k=1,2,3,4$. In $\Omega_1\cup\Omega_3$, the orbits of $\mathbf{B}$ are timelike. In $\Omega_2\cup\Omega_4$, the orbits are spacelike. When $n=2$, $\mathcal{S}=\Omega_1\cup\Omega_3$, $\mathbf{T}^{\uparrow}=\Omega_2$, $\mathbf{T}^{\downarrow}=\Omega_4$. 

We come back to the general case. To simplify computations and notations, we call $\chi:\Omega\subset \R^k\to \R^k$ the \textit{position vector}. Define the continuous map \[\pi:\L^n\to[0,+\infty), \quad  \pi(x)=\sqrt{\vert g(x,x)\vert}, \ \textrm{also}\ 
\pi=\sqrt{\vert g(\chi,\chi)\vert}.\] 
Outside the light cone, $\pi$ is smooth. Given $x\in \L^n\backslash\mathcal{C}$, if we put $\tilde{\ep}=\mathrm{sign}(g(x,x))=\pm 1$, 
\begin{equation}\label{projectionproperties}
\grad\pi(x)=\tilde{\ep}\frac{\chi(x)}{\pi(x)}, \quad 
\vert \grad \pi(x) \vert^2=\tilde{\ep}, \quad \mathrm{div}(\grad\pi )(x)= \frac{\tilde{\ep}}{\pi(x)}. 
\end{equation} 
This projection can be restricted to $\Omega_i$, $\mathcal{S}$, $\mathbf{T}^{\uparrow}$, $\mathbf{T}^{\downarrow}$, obtaining the necessary pseudo-Riemannian submersions onto $(0,+\infty)$. We explain the curves and diffeomorphisms of \eqref{equ:decom} in some cases, and the other ones are left to the reader.

$\star$ Decomposition of $\mathbf{T}^{\uparrow}$. 
The curve is $\beta:(0,+\infty)\to \mathbf{T}^{\uparrow}$, $\beta(s)=(0,\ldots,0,s)$, and the map is $\phi:(0,+\infty)\times SO(n-1,1) \to \mathcal{C}^{\uparrow}$, $\phi(s,A)=\beta(s)A$. Also, $\tilde{\ep}=-1$. 

$\star$ Decomposition of $\mathcal{S}$. The curve is now $\beta:(0,+\infty)\to\mathcal{S}$, $\beta(s)=(s,0,\ldots,0)$, and the map is $\phi:(0,+\infty)\times SO(n-1,1) \to \mathcal{C}^{\uparrow}$, $\phi(s,A)=\beta(s)A$. Also, $\tilde{\ep}=+1$. 
\begin{remark}  Given a domain (open and connected) $\Omega\subset \L^n\backslash\mathcal{C}$, such that any of the above decompositions make sense, then $\pi:\Omega\to (0,+\infty)$. By \eqref{projectionproperties}, function $h:(0,+\infty)\to\R$ is $h(s)=\tilde{\ep}/s$.  
\end{remark}

It is important to remember that we are studying $\L^{n+1}=\L^n\times\R$, so that $\ep'=+1$. Therefore, with our previous considerations, equation \eqref{keyODE} reduces to 
\begin{equation} \label{caseOmega1}
f''(s)=\big(\tilde{\ep}+f'(s)^2\big)\left(1-\frac{f'(s)}{s}\right). 
\end{equation}
where $\tilde{\ep}$ depends of the type of the orbits of $\textbf{B}$, i.e. of the region we are considering
\begin{examp}\label{ZWing} \normalfont Assume $\tilde{\ep}=1$, i.e. we consider solutions in the domain $S$. Then \eqref{keyODE} becomes
\begin{equation} \label{caseOmega1}
f''(s)=\big(1+f'(s)^2\big)\left(1-\frac{f'(s)}{s}\right). 
\end{equation}
This is the very same ODE as in the rotationally symmetric case in $\R^3$. By \cite{CSS}, there are two types of solutions. \\
$\bullet$ \textit{Type Z:} Firstly, we recall the only solution $f_1$ such that $f_1(0)=0$ and $f_1'(0)=0$. The traslating soliton will be the corresponding graph 
\[ \Phi:\Omega \to \Omega\times\R\subset \L^{n+1}, \quad \Phi(x)=\big(x,f(\pi(x))\big).\]
When $n=2$, there will be two twin surfaces constructed like this, one in $\Omega_1$ and one in $\Omega_3$. For $n\geq 3$, since $\Omega=\mathcal{S}$ is connected, there is just one hypersurface. 

\noindent $\bullet$ \textit{Wing-like:} Secondly, as in \cite{CSS}, (also, recall Corollary \ref{oide}), there is a family of unextendable, spacelike profile curves $\alpha:\R\to \{(x,y,t)\in \L^3  : y=0, x>0 \}$. By denoting the $3\times 3$ matrix $\hat{A}_{\theta}=\left(\begin{smallmatrix} A_{\theta} & 0 \\ 0 & 1 \end{smallmatrix}\right)$, 
the translating soliton is then
\[ \Phi: \R^2\to \L^3, \quad \Phi(\theta,s)=\alpha(s)\hat{A}_{\theta}. 
\]$\Box$
\end{examp}
\begin{examp}\label{tildeepisminusone} \normalfont 
For $\ep'=1$ and $\tilde{\ep}=-1$, \eqref{keyODE} becomes 
\begin{equation}\label{caseOmega2}
f''(s) = \big(-1+f'(s)^2\big)\left( 1 +\frac{f'(s)}{s}\right). 
\end{equation}
By the easy change $q(s)=-f(s)$, we transform this problem in
\[ q''(s) = \big(1-q'(s)^2\big)\left( 1 -\frac{q'(s)}{s}\right). 
\]
By Section \ref{soluODE}, there are $7$ types of solutions to this differential equation. Once we have a solution $(f=-q)$, we construct our profile curve and the associated translating soliton. As $f'=-q'$, the solutions producing timelike translating solitons are those satisfying $\vert f'(s)\vert <1$.  So, for each solution as in Lemma \ref{moduloxmenor1}, we construct a timelike translating soliton.
\end{examp}

\begin{examp} \textit{The Hybrid Translator:} \normalfont 
$\bullet$ Case $n=2$. We call $f_1$ the solution to \eqref{caseOmega1} in $\R$, and $f_2$ the solution to \eqref{caseOmega2} in $\R$ such that $f_i(0)=0$ and $f_i'(0)=0$, $i=1,2$. Recall that they are even functions and analytical.
As such, their derivatives of odd order are $f_i^{(2k-1)}(0)=0$, $k\geq 1$. 
As a result, $f_1(s)=\sum_{k=1}^{\infty} \frac{f_1^{(2k)}(0)}{(2k)!}s^{2k}$. Then, it makes sense to define $\hat{f}:\R\to\R$, 
$\hat{f}(s)=f_1(is)$, where $i=\sqrt{-1}$. By simple computations,
\[\hat{f}''(s) = 
- \big(1 - (\hat{f}'(s))^2 \big)\Big(1 + \frac{\hat{f}'(s)}{s}\Big).\]
Hence, $\hat{f}$ is a solution to equation \eqref{caseOmega2}, and consequently $\hat{f}=f_2$. By comparing the derivatives of $f_1$ and $f_2$,  
$f_1^{(k)}(0)=f_2^{(k)}(0)=0$, if $k$ is odd, since $f_1$ and $f_2$ are even, and $f_1^{(4k+2)}(0)=-f_2^{(4k+2)}(0)$ and $f_1^{(4k)}(0)=f_2^{(4k)}(0)$ for all $k\geq 0$. Next, we can define 
\[ u:\R^2\rightarrow \R, \quad u(x,y) = \left\{
\begin{array}{cl} f_1\left(\sqrt{x^2-y^2}\right),  & (x,y)\in \Omega_1\cup\Omega_3, \\
0,  & (x,y)\in\partial\Omega_i, i=1,2,3,4,\\
f_2\left(\sqrt{y^2-x^2}\right),  & (x,y)\in\Omega_2\cup\Omega_4.
\end{array}
\right.
\]
Note that we get immediately $u\in C^0(\R^2)\cap C^{\infty}(\bigcup_{k=1}^4\Omega_k)$.  We want to prove that $u\in C^{\infty}(\R^2)$. To do so, we now just need to prove the following
\begin{lemma}
Let $f_1$, $f_2$ be functions in $C^{2m}(\R)$, such that $f_1^{(k)}(0)=f_2^{(k)}(0)=0$, if $k$ is odd, $f_1^{(k)}(0)=(-1)^{\frac{k}{2}}f_2^{(k)}(0)$, if $k$ is even. Then the function $u$ defined as above is in $C^{m}(\R^2)$.
\end{lemma}
\begin{proof}
We prove the statement by induction over $m$. The case $m=0$ is trivially satisfied, since $f_1(0)=f_2(0)$. Moreover let $g_1(z)=\frac{f'_1(z)}{z}$, $g_2(z)=\frac{f'_2(z)}{z}$. We have
\begin{gather*}
\partial_xf_1(\sqrt{x^2-y^2})=xg_1(\sqrt{x^2-y^2}),\quad \partial_xf_2(\sqrt{y^2-x^2})=-xg_2(\sqrt{y^2-x^2}), \\
\partial_yf_1(\sqrt{x^2-y^2})=-yg_1(\sqrt{x^2-y^2}),\quad \partial_yf_2(\sqrt{y^2-x^2})=yg_2(\sqrt{y^2-x^2}).
\end{gather*} 
Obviously $g_1, g_2\in C^{2m-2}$ since $f_1'(0)=f'_2(0)=0$, and $g_1^{(4k+2)}(0)=-g_2^{(4k+2)}(0)$ and $g_1^{(4k)}(0)=g_2^{(4k)}(0)$ for all $k\geq 0$. Then by the induction hypothesis, the function \[ v:\R^2\rightarrow \R, \quad v(x,y) = \left\{
\begin{array}{cl} g_1\left(\sqrt{x^2-y^2}\right),  & (x,y)\in \Omega_1\cup\Omega_3, \\
0,  & (x,y)\in\partial\Omega_i, i=1,2,3,4,\\
g_2\left(\sqrt{y^2-x^2}\right),  & (x,y)\in\Omega_2\cup\Omega_4.
\end{array}
\right.
\]
is in $C^{m-1}(\R^2)$. Hence $\partial_xu$ and $\partial_yu$ extend to $C^{m-1}(\R^2)$, and finally $u\in C^{m}(\R^2)$.
\end{proof}
We point out that the curve joining each two adjacent pieces is a lightlike straight line. Thus,  it is possible to consider two or three contiguous pieces, and their gluing straight lines. We cannot consider the other straight half-lines, since the boundary would be lightlike. In other words, we can choose to glue either two, three or four adjacent pieces, to obtain translating solitons. \\

$\bullet$ Case $n\geq 3$. $\L^{n}$ is isometric to the Lorentzian product $\mathbb{R}^{n-1}\times_{-1}\R$, $n\geq 3$, with the standard metric as at the beginning of this section. If we call $g_o$ the standard Riemannian metric on $\R^{n-1}$,  its norm is $\| x\|=\sqrt{g_o(x,x)}$, for any $x\in \R^{n-1}$. The function $u$ of the Hybrid Translator can be easily extended to $\L^n$ as follows:
\[ \tilde{u}:\R^{n-1}\times_{-1}\R=\L^{n}\to \L^{n+1}, \quad (x,y)\mapsto \tilde{u}(x,y):=u(\|x\|, y). 
\]
There is no confusion if we also call its graph the \textit{Hybrid Translator}. 

In addition, take $A\in SO^{\uparrow}(n-1,1)$, and decompose it as $A=\left(\begin{smallmatrix} B & d^t \\ c & a \end{smallmatrix}\right)$ for suitable $B\in\mathcal{M}_{n-1}(\R)$, $c,d\in\R^{n-1}$, $a\in \R$. With this, given $(x,y)\in \R^{n-1}\times_{-1}\R=\L^{n}$, taking $(z,t)=(z,y)A$, then $\|z\|^2-t^2=\|x\|^2-y^2$. This means that $\tilde{u}$ is invariant by $SO^{\uparrow}(n-1,1)$. 

Sprunk and Xiao, \cite{SX}, proved that any entire translating soliton in $\R^3$ must be convex, but this is not the case for our hybrid example, because $f_1f_2<0$. $\Box$
\end{examp}
\begin{theorem} Let $M$ be a $SO^{\uparrow}(n-1,1)$ invariant, timelike, translating soliton in $\L^{n+1}$, $n\geq 2$. Then, up to isometries, $M$ is an open subset of one of the following examples:
\begin{enumerate}
\item A translating soliton of type $Z$ or \textit{Wing-like}.
\item Given $w$ any of the three types of solutions in Lemma \ref{moduloxmenor1}, take $f=-\int w$. 
\item The Hybrid Translator.
\end{enumerate}
\end{theorem}
\begin{proof} We will use Theorem \ref{equiv_ODE}. As we regard $\L^{n+1}=\L^n\times\R$, we need $\ep'=+1$.  Recall $\L^n=\mathcal{C}\cup\mathcal{S}\cup\mathbf{T}^{\uparrow}\cup\mathcal{C}^{\downarrow}$, and each subset is invariant by the action of $SO^{\uparrow}(n-1,1)$ by isometries. Only on $\mathcal{S}\cup\mathbf{T}^{\uparrow}\cup\mathcal{C}^{\downarrow}$ the action is proper, and we can obtain  decompositions as in \eqref{equ:decom}. The value of $\tilde{\ep}$ depends on the chosen domain. The orbits  of the action are well-known. First, one finds the  lightcone. For $n=2$, they are (Euclidean) hyperbolas. For $n\geq 3$, they are hypersurfaces with non-zero constant sectional curvature (namely, real hyperbolic spaces for any $n\geq 3$ and Anti-de Sitter spaces for $n\geq 4$).

$\bullet$ For $\mathbf{T}^{\uparrow}$ (resp. $\mathbf{T}^{\downarrow}$), the curve $\beta:(0,+\infty)\to \mathbf{T}^{\uparrow}$, $\beta(s)=(0,\ldots,0,s)$ is timelike (resp. $\beta(s)=(0,\ldots,0,-s)$). Then, $\tilde{\ep}=-1$ and therefore we recall Example \ref{tildeepisminusone}, obtaining another 3 timelike translating solitons. 

$\bullet$ For  $\Omega=\mathcal{S}$, the curve $\beta:(0,+\infty)\to\mathcal{S}$, $\beta(s)=(s,0,\ldots,0)$ is spacelike. Then, $\tilde{\ep}=+1$. The differential equation  \eqref{keyODE} becomes
\begin{equation}\label{caseCsomething}
f''(s) = \big(1+f'(s)^2\big)\left( 1 -\frac{f'(s)}{s}\right). 
\end{equation}
Now, we recall Examples \ref{ZWing}. One should bear in mind that for $n=2$, there will two twin surfaces, in $\Omega_1\times\R$ and in $\Omega_3\times\R$, since $\mathcal{S}$ is not connected. 

$\bullet$ It remains to study whether the solutions \textit{touching} the lightcone can be extended smoothly. Since the lightcone  has a degenerate metric, it has to be strictly contained in the hypersurface. In other words, we wonder if it is possible to \textit{glue} two solutions, one in $\mathbf{T}^{\uparrow}$ and one in $\mathcal{S}$ (alternatively, in $\mathbf{T}^{\downarrow}$). By regarding the extended action of $SO^{\uparrow}(n-1,1)$ on $\L^{n+1}$, if any, the tangent plane at the origing must be $\{x\in\L^{n+1}=\L^n\times \R : x_{n+1}=0\}$. This means that the derivatives of the functions must be zero. Then, we finish the proof by recalling the Hybrid Translator. 
\end{proof}

By another method, in \cite{coreano}, the author obtained (essentially) the three \textit{spacelike} types, but not the timelike cases.  Again, we also recover the following result from \cite{coreano}. 
\begin{theorem} Up to isometries, any $SO^{\uparrow}(n-1,1)$ invariant, spacelike, translating soliton in $\L^{n+1}$, $n\geq 2$, is an open subset of a hypersurface generated by the graph of any of Cases V, VI or VII in Lemma \ref{clasificacion-timelike}. 
\end{theorem}
\begin{remark}\normalfont As in the proof of Theorem \ref{timelike}, Type IV from Lemma \ref{caseiv} and Type VII can be joined to provide a single profile curve. 
\end{remark}
\begin{corollary} The only $SO^{\uparrow}(n-1,1)$-invariant, entire translating soliton in $\L^{n+1}$, $n\geq 2$, is the hybrid translator. 
\end{corollary}

\section*{Acknowledgements}

The second author has been partially financed by the Spanish Ministry of Economy and Competitiveness and European Regional Development Fund (ERDF), project MTM2016-78807-C2-1-P.  The second author also belongs to the Excellence Scientific Unit 'Science in the Alhambra' of Granada University, ref. UCE-PP2018-01. 
Both authors would like to thank the London Mathematical Society and Imperial College London, since this was partially supported by a Scheme 4 Grant of the LMS.

\end{document}